\documentclass[12pt]{amsart}
\usepackage{amscd,amsmath,amsthm,amssymb}
\usepackage[left]{lineno}
\usepackage{pstcol,pst-plot,pst-3d}
\usepackage{color}
\usepackage{pstricks}
\usepackage{stmaryrd}
\usepackage[utf8]{inputenc}
\usepackage{pstricks-add}
\usepackage[dvips]{graphicx}
\usepackage{epstopdf}
\usepackage{graphicx}
\newpsstyle{fatline}{linewidth=1.5pt}
\newpsstyle{fyp}{fillstyle=solid,fillcolor=verylight}
\definecolor{verylight}{gray}{0.97}
\definecolor{light}{gray}{0.9}
\definecolor{medium}{gray}{0.85}
\definecolor{dark}{gray}{0.6}

 \def\NZQ{\mathbb}               
 \def\NN{{\NZQ N}}

 \def\frk{\mathfrak}               

 \def\mm{{\frk m}}

 \def\G{{\mathcal G}}

  \def\Mc{{\mathcal M}}

 \def\xb{{\mathbf x}}

 \def\opn#1#2{\def#1{\operatorname{#2}}} 
 \opn\chara{char} \opn\length{\ell} \opn\pd{pd} \opn\rk{rk}
 \opn\projdim{proj\,dim} \opn\injdim{inj\,dim} \opn\rank{rank}
 \opn\depth{depth} \opn\grade{grade} \opn\height{height}
 \opn\embdim{emb\,dim} \opn\codim{codim}
 
 \opn\Tr{Tr} \opn\bigrank{big\,rank}
 \opn\superheight{superheight}\opn\lcm{lcm}
 \opn\trdeg{tr\,deg}
 \opn\reg{reg} \opn\lreg{lreg} \opn\ini{in} \opn\lpd{lpd}
 \opn\size{size} \opn\sdepth{sdepth}
 \opn\link{link}\opn\fdepth{fdepth}\opn\lex{lex}
 \opn\tr{tr}
 \opn\type{type}
 \opn\gap{gap}
 \opn\arithdeg{arith-deg}
 \opn\div{div} \opn\Div{Div} \opn\cl{cl} \opn\Cl{Cl}
 \opn\Spec{Spec} \opn\Supp{Supp} \opn\supp{supp} \opn\Sing{Sing}
 \opn\Ass{Ass} \opn\Min{Min}\opn\Mon{Mon}
 \opn\Ann{Ann} \opn\Rad{Rad} \opn\Soc{Soc}
 \opn\Im{Im} \opn\Ker{Ker} \opn\Coker{Coker} \opn\Am{Am}
 \opn\Hom{Hom} \opn\Tor{Tor} \opn\Ext{Ext} \opn\End{End}
 \opn\Aut{Aut} \opn\id{id}
 
 \opn\nat{nat}
 \opn\pff{pf}
 \opn\Pf{Pf} \opn\GL{GL} \opn\SL{SL} \opn\mod{mod} \opn\ord{ord}
 \opn\Gin{Gin} \opn\Hilb{Hilb}\opn\sort{sort}
 \opn\PF{PF}\opn\Ap{Ap}
 \opn\mult{mult}
 \opn\bight{bight}
 \opn\aff{aff}
 \opn\relint{relint} \opn\st{st}
 \opn\lk{lk} \opn\cn{cn} \opn\core{core} \opn\vol{vol}  \opn\inp{inp} \opn\nilpot{nilpot}
 \opn\link{link} \opn\star{star}\opn\lex{lex}\opn\set{set}
 \opn\width{wd}
 \opn\Fr{F}
 \opn\QF{QF}
 \opn\G{G}
 \opn\type{type}\opn\res{res}
 \opn\conv{conv}
 \opn\Ind{Ind}
 \opn\gr{gr}
 
 \def\pot#1#2{#1[\kern-0.28ex[#2]\kern-0.28ex]}

 \opn\dirlim{\underrightarrow{\lim}}
 \opn\inivlim{\underleftarrow{\lim}}
 \let\union=\cup

 \let\iso=\cong

 \let\Dirsum=\bigoplus
 
 \let\to=\rightarrow
 
 \def\Implies{\ifmmode\Longrightarrow \else
         \unskip${}\Longrightarrow{}$\ignorespaces\fi}
 \def\implies{\ifmmode\Rightarrow \else
         \unskip${}\Rightarrow{}$\ignorespaces\fi}
 \def\iff{\ifmmode\Longleftrightarrow \else
         \unskip${}\Longleftrightarrow{}$\ignorespaces\fi}

 \let\:=\colon
 \newtheorem{Theorem}{Theorem}[section]
 \newtheorem{Lemma}[Theorem]{Lemma}
 \newtheorem{Corollary}[Theorem]{Corollary}
 \newtheorem{Proposition}[Theorem]{Proposition}
 \newtheorem{Remark}[Theorem]{Remark}
 \newtheorem{Remarks}[Theorem]{Remarks}

 \newtheorem{Definition}[Theorem]{Definition}

 \let\epsilon\varepsilon
 \let\kappa=\varkappa
 \def\qed{\ifhmode\textqed\fi
       \ifmmode\ifinner\quad\qedsymbol\else\dispqed\fi\fi}
 \def\textqed{\unskip\nobreak\penalty50
        \hskip2em\hbox{}\nobreak\hfil\qedsymbol
        \parfillskip=0pt \finalhyphendemerits=0}
 \def\dispqed{\rlap{\qquad\qedsymbol}}
 
 \opn\dis{dis}
 \def\pnt{{\raise0.5mm\hbox{\large\bf.}}}
 
 \opn\Lex{Lex}

\begin{document}

\title {Sortable simplicial complexes and $t$-independence ideals of proper interval graphs}

\author {J\"urgen Herzog, Fahimeh Khosh-Ahang, Somayeh Moradi and  Masoomeh Rahimbeigi}

\address{J\"urgen Herzog, Fachbereich Mathematik, Universit\"at Duisburg-Essen, Campus Essen, 45117
Essen, Germany} \email{juergen.herzog@uni-essen.de}

\address{Fahimeh Khosh-Ahang, Department of Mathematics, School of Science, Ilam University,
P.O.Box 69315-516, Ilam, Iran}
\email{f.khoshahang@ilam.ac.ir}

\address{Somayeh Moradi, Department of Mathematics, School of Science, Ilam University,
P.O.Box 69315-516, Ilam, Iran}
\email{so.moradi@ilam.ac.ir}

\address{Masoomeh Rahimbeigi, Department of Mathematics, University of Kurdistan, Post
Code 66177-15175, Sanandaj, Iran}
\email{rahimbeigi$_{-}$masoome@yahoo.com}

\dedicatory{ }

\begin{abstract}
We introduce the notion of sortability and $t$-sortability for a simplicial complex and study the graphs for which their independence complexes are either sortable or $t$-sortable. We show that the proper interval graphs are precisely the graphs whose independence complex is sortable. By using this characterization, we show that the ideal generated by all squarefree monomials corresponding to independent sets of vertices of $G$ of size $t$ (for a given positive integer $t$) has the strong persistence property, when $G$ is a proper interval graph. Moreover, all of its powers have linear quotients.
\end{abstract}

\thanks{}




\maketitle

\setcounter{tocdepth}{1}

\section*{Introduction}

The notion of strong persistence property for an ideal in a Noetherian ring $R$ has been defined in \cite{HQ}. It is known that any monomial ideal with the strong persistence property has the persistence property (see \cite{HQ}). Although finding ideals with the strong persistence property is of great interest, but there is not much known about them.
Few classes of monomial ideals are known to possess this property. Polymatroidal ideals (\cite{HRV}) and  edge ideals of graphs (\cite{MMV})  are some of these families.
In this paper, we introduce a new class of monomial ideals associated to proper interval graphs with the strong persistence property. To this aim, we introduce the notion of a sortable simplicial complex and show that the independence complex of a graph $G$ is sortable if and only if $G$ is a proper interval graph. Using this characterization, we obtain some algebraic properties of the $t$-independence ideal $I_t(G)$ generated by all squarefree monomials corresponding to independent sets of vertices of $G$ of size $t$, when $G$ is a proper interval graph. It is proved that this ideal has the strong persistence property.
Moreover, when $G$ is a proper interval graph or an $n$-cycle, it is shown that the toric ring $K[u\: u\in \mathcal{G}(I_t(G))]$ over the field $K$ is Koszul and a normal Cohen--Macaulay domain.

We recall some definitions and notation that are needed in the sequel.
Let $G$ be any finite simple graph on the vertex set $V$. A subset $F\subseteq V$ is called an {\em independent set} of $G$ if it contains no edge of $G$. The set of all independent sets of $G$ forms a simplicial complex $\Delta(G)$, which is called the {\em independence complex} of $G$. For a graph $G$ on the vertex set $[n]$, a subset $A\subseteq [n]$ is called an {\em interval in $G$}, if $A=\{r,r+1,\ldots,s\}$ for some $r\leq s$. The set of all vertices adjacent to a vertex $v$ in $G$ is denoted by $N_G(v)$ and by $N_G[v]$ we mean $N_G(v)\cup\{v\}$. The path graph and the cycle graph with $n$ vertices are denoted by $P_n$ and $C_n$, respectively.

A graph $G$ is called an {\em interval graph}, if one can label its vertices with some intervals on the real line so that two vertices are adjacent in $G$, when the intersection of their corresponding intervals is non-empty. A {\em proper interval graph} is an interval graph such that no interval properly contains another. Proper interval graphs are well studied in the literature, see for example \cite{BW,G,LO,R1}. In this paper we give another characterization of these graphs in terms of the sortability of their independence complexes.

Let $S=K[x_{1},\ldots, x_{n}]$ be a polynomial ring over a field $K$ and $u$ and $v$ be two monomials of degree $d$ in $S$. Write $uv=x_{i_1}x_{i_2}\cdots x_{i_{2d}}$ with $1\leq i_1\leq i_2\leq \cdots\leq i_{2d}$, and
set $u^\prime=x_{i_1}x_{i_3}\cdots x_{i_{2d-1}}$ and $ v^\prime =x_{i_{2}}x_{i_4}\cdots x_{i_{2d}}.$ The pair
$(u^\prime,v^\prime)$ is called the \emph{sorting} of $(u,v)$ and is denoted by $\sort(u,v).$ Note that if $u$ and $v$ are squarefree, then $u'$ and $v'$ are squarefree, as well. The pair $(u,v)$ is called a {\em sorted pair},  if $\sort(u,v)=(u,v)$.  Otherwise,  $(u,v)$ is called an {\em unsorted pair}. Let $S_d$ be the set of all monomials of degree $d$ in $S$. A subset $\Mc\subset S_d$ is called \emph{sortable} if $\sort(u,v)\in \Mc\times \Mc$ for all $(u,v)\in \Mc\times \Mc.$  We say that a monomial ideal $I$ is sortable, if it is generated in a single degree and $\mathcal{G}(I)$ is a sortable set of monomials, where  $\mathcal{G}(I)$ is the set of minimal monomial generators of $I$.

The paper proceeds as follows. In Section \ref{secone}, we introduce and study sortable and $t$-sortable simplicial complexes. As one of the main results of this section, we give a new characterization of proper interval graphs by means of sortability concept (see Theorem \ref{obvious}). Moreover, we prove that the  independence complex of any cycle graph is $t$-sortable. In Section \ref{sectwo}, we consider the $t$-independence ideals of proper interval graphs and prove that these ideals satisfy the $\ell$-exchange property and consequently the strong persistence property. Finally we show that for any ideal in this class, all of its powers have linear quotients and hence linear resolutions.

\section{Sortable simplicial complexes}
\label{secone}

Let $\Delta$ be a (finite) simplicial complex on the vertex set $V(\Delta)\subset \NN$. For any finite set $F\subset \NN$, we associate with $F$ the monomial $\xb^F=\prod_{i\in F}x_i$.

Given two faces $F,G\in \Delta$ with $|F|=r$ and $|G|=s$ we write
\[
\xb^F\xb^G=x_{i_1}x_{i_2}\cdots x_{i_{r+s}} \quad \text{with}\quad i_1\leq i_2\leq \cdots \leq i_{r+s}.
\]
We define the {\em sorting operator} as follows:
\[
\sort(F,G)=(F',G'),
\]
where $F'=\{i_k:\ 1\leq k\leq r+s, k\textrm{ is odd}\}$ and $G'=\{i_k:\ 1\leq k\leq r+s, k\textrm{ is even}\}$.

Notice that $|F'|=|G'|$ if $r+s$ is even, and $|F'|=|G'|+1$ if $r+s$ is odd.

\begin{Definition}
{\em Let $\Delta$ be a finite simplicial complex with $V(\Delta)\subset \NN$.  Then  $\Delta$ is {\em sortable with respect to the given labeling} on $V(\Delta)$, if  for any $F,G\in \Delta$, one has $\sort(F,G)\in \Delta\times\Delta$. Moreover,  $\Delta$ is {\em sortable}, if it is sortable with respect to some suitable labeling with integers on $V(\Delta)$.}
\end{Definition}

A weaker property than sortability which is called $t$-sortability is defined as follows.
\begin{Definition}
{\em Let $t$ be a positive integer. A finite simplicial complex $\Delta$ with $V(\Delta)\subset \NN$ is called {\em $t$-sortable with respect to the given labeling} on $V(\Delta)$, if  for any $F,G\in \Delta$ with $|F|=|G|=t$ we have $\sort(F,G)\in \Delta\times\Delta$. Moreover,  $\Delta$ is {\em $t$-sortable}, if it is $t$-sortable with respect to some suitable labeling with integers on $V(\Delta)$.}
\end{Definition}

Note that if $\Delta$ is $t$-sortable and $F,G\in \Delta$ such that $|F|=|G|=t$ and $\sort(F,G)=(F',G')$, then $|F'|=|G'|=t$.

We have the following simple observations.

\begin{Remarks}
\label{restriction}
{\em Let $\Delta$ be a finite simplicial complex with $V(\Delta)\subset \NN$.

 \begin{enumerate}
\item[(i)] If $\Delta$ is sortable  ($t$-sortable), then for any $T\subset V(\Delta)$, the simplicial complex $\Delta_T=\{F\in \Delta \:\; F\subset T\}$ is also sortable  ($t$-sortable).
\item[(ii)]   If $\Delta$ is sortable, then it is $t$-sortable for any positive integer $t$.
\item[(iii)] The converse of (ii) does not hold in general. Indeed for any $n\geq 4$, $\Delta(C_n)$ is $t$-sortable for all $t$ and is not sortable (see Theorem \ref{obvious} and Proposition \ref{standard}).
\end{enumerate}
}
\end{Remarks}

Recall that if $\Delta_1$ and $\Delta_2$ are simplicial complexes on disjoint sets of vertices, then the {\em join} of $\Delta_1$ and $\Delta_2$ denoted by  $\Delta_1*\Delta_2$ is a simplicial complex on the vertex set $V(\Delta_1*\Delta_2)=V(\Delta_1)\union V(\Delta_2)$ defined as $\Delta_1*\Delta_2=\{F\union G\:\; F\in \Delta_1 \text{ and } G\in \Delta_2\}$.


\begin{Proposition}
\label{join}
Let $\Delta_1$ and $\Delta_2$ be simplicial complexes on disjoint sets of vertices.
Then $\Delta_1*\Delta_2$ is sortable, if and only if $\Delta_1$ and $\Delta_2$ are sortable.
\end{Proposition}

\begin{proof}
Let $\Delta_1$ and $\Delta_2$ be sortable. One may consider sorting labelings  on the vertices of $\Delta_1$ and $\Delta_2$ in $\NN$ such that $i<j$ for all $i\in V(\Delta_1)$ and $j\in V(\Delta_2)$.
Consider two elements $F_1\cup F_2$ and $G_1\cup G_2$ in $\Delta_1*\Delta_2$ with $F_1,G_1\in \Delta_1$ and $F_2,G_2\in \Delta_2$. Let $(F'_1,G'_1)=\sort(F_1,G_1)$ and $(F'_2,G'_2)=\sort(F_2,G_2)$. One can see that if $|F_1|+|G_1|$ is even, then $\sort(F_1\cup F_2,G_1\cup G_2)=(F'_1\cup F'_2,G'_1\cup G'_2)$ and if $|F_1|+|G_1|$ is odd, then
$\sort(F_1\cup F_2,G_1\cup G_2)=(F'_1\cup G'_2,G'_1\cup F'_2)$. Since $F'_1,G'_1\in \Delta_1$ and $F'_2,G'_2\in \Delta_2$, $\Delta_1*\Delta_2$ is sortable.

Conversely, let $\Delta_1*\Delta_2$ be sortable. For any two faces $F,G\in \Delta_1$, since $F,G\in \Delta_1*\Delta_2$, we have $F',G'\in \Delta_1*\Delta_2$, where $(F',G')=\sort(F,G)$. Note that $F',G'\subseteq V(\Delta_1)$ and $ V(\Delta_1)\cap V(\Delta_2)=\emptyset$. This implies that $F',G'\in \Delta_1$.
By similar argument $\Delta_2$ is also sortable.
\end{proof}

Let $G$ be the disjoint union of two graphs $G_1$ and $G_2$. Then
\[
\Delta(G)=\Delta(G_1)*\Delta(G_2).
\]
Thus, we may  apply Proposition~\ref{join} and  obtain

\begin{Corollary}
\label{components}
Let $G$ be a finite simple graph with vertices in $\NN$ and $G_1, \ldots,G_m$ be the connected components of $G$. Then $\Delta(G)$ is sortable, if and only if each $\Delta(G_r)$ is sortable.
\end{Corollary}

\begin{Remark}
{\em If we replace sortability by $t$-sortability in Corollary \ref{components}, the `only if' part holds by Remarks \ref{restriction}(i).
But the `if' part does not hold in general. For example, consider a graph $G$ with two connected components $G_1$ and $G_2$, where $G_1$ is a star graph on $4$ vertices and $G_2$ is a path graph on $4$ vertices. Then by CoCoA computations one can see that the defining ideal of the fiber ring of $I_3(G)=\langle x^F:\ F\in \Delta(G), |F|=3\rangle$ is not quadratic. So by Theorem \ref{crucial}, $I_3(G)$ is not a sortable ideal. Hence $\Delta(G)$ is not $3$-sortable. But it is easy to see that $\Delta(G_1)$ and $\Delta(G_2)$ are $3$-sortable.}
\end{Remark}

\medskip

The following lemma states some equivalent conditions for a graph to be proper interval. We use this result in Theorem \ref{obvious} to characterize the graphs whose independence complexes are sortable.
\begin{Lemma}\label{def}
For a graph $G$ on the vertex set $[n]$, the following conditions are equivalent:
\begin{enumerate}
\item[(i)] For all $i<j$, $\{i, j\}\in E(G)$ implies that the induced subgraph of $G$ on $\{i, i+1, \dots ,j\}$ is a clique.
\item[(ii)]  For all $1\leq i \leq n$, $N_{G^i}[i]$ is both a clique and an interval, where $G^i$ is the induced subgraph of $G$ on $\{i,i+1, \dots, n\}$.
\item[(iii)] For all $1\leq i \leq n$, $N_{G_i}[i]$ is both a clique and an interval,  where $G_i$ is the induced subgraph of $G$ on $\{1,2, \dots,i\}$.
\item[(iv)] For all  $1\leq i \leq n$, $N_{G}[i]$ is an interval.
\item[(v)]  $G$ is a proper interval graph.
\end{enumerate}
\end{Lemma}
\begin{proof}
(i) $\Rightarrow$ (ii) Suppose that $1\leq i \leq n$ and  $j$ is the largest integer such that  $j\in N_{G^i}(i)$. Then $\{i, j\}\in E(G)$. So, by (i) the induced subgraph of $G$ on $\{i, i+1, \dots ,j\}$ is a clique. This proves the result.

(ii) $\Rightarrow$ (iii) Suppose that $1\leq i \leq n$ and  $j$ is the least integer such that  $j\in N_{G_i}(i)$. Then $i\in N_{G^j}(j)$. Thus (ii) implies that  the induced subgraph of $G$ on $\{j, j+1, \dots ,i\}$ is a clique. This shows that  $N_{G_i}[i]$ is a clique and an interval as desired.

(iii) $\Rightarrow$ (iv) Suppose that $1\leq i \leq n$ and $j,k\in N_G[i]$ with $j<k$. It is enough to prove that for each integer $\ell$ between $j$ and $k$, one has $\ell\in N_G[i]$. If $i<j<k$, then the result follows from $i\in N_{G_k}[k]$ and  the assumption that $N_{G_k}[k]$ is both clique and interval. If $j<k<i$, then the result follows from $j\in N_{G_i}[i]$ and  the assumption that $N_{G_i}[i]$ is an interval. Now, assume that $j\leq i\leq k$. then $i\in N_{G_k}[k]$ and $j\in N_{G_i}[i]$. Since  $N_{G_k}[k]$ is a clique and an interval and $N_{G_i}[i]$ is an interval, the result is obtained.

(iv) $\Rightarrow$ (i) Suppose that $i<j$ and $\{i,j\}\in E(G)$. Since $N_G[i]$ is an interval and $j\in N_G[i]$, $\{i+1, \dots , j\}\subseteq N_G[i]$. Now for each $i <\ell \leq j$, since $i\in N_G[\ell]$ and $N_G[\ell]$ is an interval, $\{i, i+1, \dots , \ell-1\}\subseteq N_G[\ell]$. This shows that the induced subgraph of $G$ on $\{i,i+1, \dots ,j\}$ is a clique.

(i) $\Leftrightarrow$ (v) See \cite[Theorem 1 and Proposition 1]{LO}.
\end{proof}

Property (iii) of Lemma \ref{def}, implies that any proper interval graph has a perfect elimination ordering and hence is a chordal graph.

\begin{Theorem}\label{obvious}
Let $G$ be a graph. Then $\Delta(G)$ is sortable if and only if $G$ is a proper interval graph.
\end{Theorem}

\begin{proof}
Let $\Delta(G)$ be sortable and by contrary assume that $G$ is not proper interval. Then by Lemma \ref{def}, for any labeling on $V(G)$ there exists $i\in V(G)$
such that $N_G[i]$ is not an interval. This means that there exists $j,k\in N_G[i]$ and an integer $\ell$ with $j<\ell<k$ such that $\ell\notin N_G[i]$. If $i<\ell$, then
$\sort(\{i,\ell\},\{k\})=(\{i,k\},\{\ell\})$ and $\{i,k\}\notin \Delta(G)$, a contradiction. If $i>\ell$, then
$\sort(\{i,\ell\},\{j\})=(\{j,i\},\{\ell\})$ and $\{j,i\}\notin \Delta(G)$, which contradicts to sortability of $\Delta(G)$.


Conversely, suppose that $G$ is a proper interval graph. Then by Lemma \ref{def}, we may consider a labeling on $V(G)=[n]$ such that
for all $i<j$, $\{i, j\}\in E(G)$ implies that the induced subgraph of $G$ on $\{i, i+1, \dots ,j\}$ is a clique. Let $F_1,F_2\in \Delta(G)$ and assume that $x^{F_1}x^{F_2}=x_{i_1}x_{i_2}\cdots x_{i_{r+s}}$, where $i_1\leq i_2\leq \cdots \leq i_{r+s}$. Then $\sort(F_1,F_2)=(F'_1,F'_2)$, where $F'_1=\{i_1,i_3,\ldots,i_{r'}\}$ and $F'_2=\{i_2,i_4,\ldots,i_{s'}\}$ for some $r'$ and $s'$. By contradiction if $F'_1\notin \Delta(G)$, then $\{i_{2k-1},i_{2\ell-1}\}\in E(G)$ for some $k$ and $\ell$ with $k<\ell$. Since $i_{2k-1}\leq i_{2k}\leq i_{2\ell-1}$, by our assumption,  $\{i_{2k-1},i_{2k}\},\{i_{2k},i_{2\ell-1}\}\in E(G)$.
Note that at least two distinct vertices among $i_{2k-1},i_{2k},i_{2\ell-1}$ belong to either $F_1$ or $F_2$. This implies that either $F_1$ or $F_2$ contains an edge, a contradiction. Thus $F'_1\in \Delta(G)$. By similar argument $F'_2\in \Delta(G)$.
\end{proof}

\begin{Corollary}
\label{tree}
Let $G$  be a tree. Then $\Delta(G)$ is sortable, if and only if $G$ is a path graph.
\end{Corollary}

\begin{proof}

Since any path graph is a proper interval graph, by Theorem \ref{obvious}, $\Delta(G)$ is sortable.
Conversely, let $\Delta(G)$ be sortable and suppose that $G$ is not a path graph. Then $G$ contains an induced subgraph $H$ with three edges $\{i,j\}, \{i,l\}$ and $\{i, m\}$, for distinct vertices $i,j,l$ and $m$ of $G$. We show that $\Delta(H)$ is not sortable. Let $F=\{i\}$ and $K=\{j,l,m\}$. Then $\sort(F,K)=(F',K')$, where $|F'|=|K'|=2$ and one of $F'$ and $K'$ contains $i$. Thus we have either $F'\notin \Delta(H)$ or $K'\notin \Delta(H)$. So $\Delta(H)$ is not sortable. This contradicts to Remarks \ref{restriction}(i), noting that $\Delta(H)=\Delta(G)_{V(H)}$.
\end{proof}

By Corollaries \ref{components} and \ref{tree} one can get the following result.

\begin{Corollary}
\label{forest}
Let $G$  be a forest. Then $\Delta(G)$ is sortable, if and only if each tree of the forest is a path graph.
\end{Corollary}

The independence complex of an $n$-cycle for $n\geq 4$ is not sortable by Theorem \ref{obvious} and Lemma \ref{def}. But we still have

\begin{Proposition}
\label{standard}
$\Delta(C_n)$ with the standard labeling on $C_n$ is $t$-sortable for all $t$.
\end{Proposition}

\begin{proof}
For $n=3$, the assertion is trivial. Let $n\geq 4$, and let $A$ and $B$ be two $t$-independent sets of $C_n$ and $\sort(A,B)=(A',B')$. Note that $\Delta(C_n)=\{F\in \Delta(P_n):\  \{1,n\}\nsubseteq F\}$. So we have $A,B\in \Delta(P_n)$. Thus Corollary \ref{tree} implies that $A',B'\in \Delta(P_n)$.
Let $x^Ax^B=x_{i_1}x_{i_2}\cdots x_{i_{2t-1}}x_{i_{2t}}$ with $i_1\leq i_2\leq \cdots \leq i_{2t}$. If $i_1>1$ or $i_{2t}<n$, then $\{1,n\}\nsubseteq A'$ and  $\{1,n\}\nsubseteq B'$. Therefore $A',B'\in \Delta(C_n)$ and we are done. Now, let  $i_1=1$ and $i_{2t}=n$. Then we may write $x^Ax^B=x_1x_{i_2}\cdots x_{i_{2t-1}}x_n$ with $1\leq i_2\leq \cdots \leq i_{2t-1}\leq n$. Note that $1<i_2$ and $i_{2t-1}<n$, otherwise $A$ or $B$ would not be an independent set.  By definition $A'=\{1,i_3,\ldots,i_{2t-1}\}$ and $B'=\{i_2,\ldots, i_{2t-2},n\}$. Thus $n\notin A'$ and $1\notin B'$. Hence $\{1,n\}\nsubseteq A'$ and  $\{1,n\}\nsubseteq B'$. Therefore $A',B'\in \Delta(C_n)$ as desired.
\end{proof}

\section{Algebraic properties of $t$-independence ideals of proper interval graphs}
\label{sectwo}

For a graph $G$ on the vertex set $[n]$, the $t$-independence ideal of $G$, denoted by $I_t(G)$, is defined to be the ideal generated by all monomials $u=x_{i_1}x_{i_2}\cdots x_{i_t}$ for which $\{i_1,i_2,\ldots,i_t\}$ is a $(t-1) $-face of $\Delta(G)$. The $t$-independence ideal of $G$ is in fact the $t$-clique ideal of $G^c$. The class of $t$-clique ideals was introduced by Moradi \cite{M} and had been further studied in \cite{KM} and \cite{MRFS}.
In this section we consider the $t$-independence ideal of proper interval graphs and show that they have some nice algebraic properties.

The following result,  which we quote from \cite{EH},  will be of crucial importance in what follows.
Let $I$ be a sortable monomial ideal,  $A=K[u\: u\in \mathcal{G}(I)]$ and $T=K[y_u\: u\in \mathcal{G}(I)]$ be the polynomial ring over the field $K$ in the variables $y_u$ with $u\in \mathcal{G}(I)$.
We let $L$ be the kernel of the $K$-algebra homomorphism from $T$ to $A$ with $y_u\mapsto u$ for $u\in \mathcal{G}(I)$.

Notice that if $(u,v)$ is an unsorted pair and $(u',v')=\sort(u,v)$, then $y_uy_v-y_{u'}y_{v'}\in L$ and $y_uy_v-y_{u'}y_{v'}\neq 0$, unless $(u',v')=(v,u)$. Relations of this form are called {\em sorting relations}.

\begin{Theorem}
\label{crucial}
There exists a monomial order $<'$ on $T$, called sorting order, such that for each non-zero sorting relation  $y_uy_v-y_{u'}y_{v'}$ the monomial $y_uy_v$ for the unsorted pair $(u,v)$ is the leading term. Moreover, the set of  sorting relations forms a Gr\"obner basis of $L$.
\end{Theorem}

\begin{Corollary}
\label{sortable*}
Let $\Delta$ be a $t$-sortable simplicial complex. Then  the toric ring $K[x^F\: F\in \Delta , |F|=t]$ is Koszul and a normal Cohen--Macaulay domain. In particular, when $G$ is a proper interval graph or an $n$-cycle, then $K[u\: u\in \mathcal{G}(I_t(G))]$ is Koszul and a normal Cohen--Macaulay domain.
\end{Corollary}

\begin{proof} Since the ideal  $I=\langle x^F:\ F\in \Delta , |F|=t\rangle$ is sortable, by Theorem~\ref{crucial},    the defining ideal $L$ of $A=K[u\: u\in \mathcal{G}(I)]$ has a quadratic Gr\"obner basis with respect to the sorting order.  It follows  that $A$ is Koszul, see for example \cite[Theorem 2.28]{HHO}. Since the initial ideal of  $L$  is squarefree, by the theorem of Sturmfels \cite{St} (see also \cite[Corollary 4.26]{HHO}), $A$ is  normal. Now we apply the result of Hochster \cite[Theorem 6.3.5]{BH} which says that any normal toric ring is Cohen--Macaulay. The second statement is obtained by applying the first part on $\Delta=\Delta(G)$, Theorem \ref{obvious}, Remarks \ref{restriction}(ii) and Proposition \ref{standard}.
\end{proof}

We now consider the Rees ring of the $t$-independence ideals of proper interval graphs. To this end we recall the concept introduced in \cite{HHV},  which is  called the  $\ell$-exchange property.

Let $I\subset S=K[x_1,\ldots,x_n]$ be a monomial ideal  generated in a single degree. Then $A=K[u\: u\in \mathcal{G}(I)]$  is isomorphic to the fiber $\mathcal{R}(I)/\mm \mathcal{R}(I)$ of the Rees ring $\mathcal{R}(I)=\Dirsum_{k\geq 0}I^kt^k$, where $\mm=\langle x_1,\ldots,x_n\rangle$ is the
graded maximal ideal of $S$.  Then $\mathcal{R}(I)\iso R/J$,  where $R=S[y_u\: u\in \mathcal{G}(I)]$ and  $J$ is the kernel of the $K$-algebra homomorphism $R\to \mathcal{R}(I)$ with $x_i\mapsto x_i$ for $i=1,\ldots,n$ and $y_u\mapsto ut$ for any $u\in \mathcal{G}(I)$.

Let $A$, $T$ and $L$ be defined as before Theorem~\ref{crucial}.  We fix a monomial order $<'$ on $T$. A monomial $w \in T$  is called a {\em standard monomial} of $L$ with respect to $<'$, if  $w\not \in \ini_{<'}(L)$.

For example, if $I$ is sortable and we let $<'$ be the sorting order on $T$, then $w=y_{u_1}\cdots y_{u_N}$  is a standard monomial of $L$ with respect to $<'$ if and only if  $(u_{i}, u_{j})$ is sorted for all $i<j$.

\begin{Definition}
\label{l}
{\em Let $I$ be a monomial ideal. Then  $I$ is said to satisfy the {\em$\ell$-exchange
property} with respect to the monomial order $<$ on $T$,  if the following condition is satisfied: let  $y_{u_1}\cdots y_{u_N}$ and $y_{v_1}\cdots y_{v_N}$ be any two standard monomials of $L$ with respect to $<$  such that
\begin{enumerate}
\item[(i)]
$\deg_{x_r} (u_1\cdots u_N)=\deg_{x_r} (v_1\cdots v_N)$ for $r=1,\ldots,q-1$ with $q\leq n-1$,
\item [ (ii)]
$\deg_{x_q}(u_1\cdots u_N)<\deg_{x_q}(v_1\cdots v_N)$.
\end{enumerate}
Then there exists an integer $k$,  and an integer $q<j\leq n$ with $x_j\in\supp(u_k)$ such that  $x_qu_k/x_j\in I$.}
\end{Definition}

\begin{Proposition}
\label{lcondition}
Let $G$ be a proper interval graph on the vertex set $[n]$. Then for all $t\geq 2$,  the ideal $I_t(G)$ satisfies the $\ell$--exchange
property with respect to the sorting order.
\end{Proposition}

\begin{proof} Let $y_{u_1}\cdots y_{u_N}$ and $y_{v_1}\cdots y_{v_N}$  be standard  monomials satisfying (i) and (ii) of Definition~\ref{l}. Then $(u_i,u_j)$ and $(v_i,v_j)$ are sorted for any $i<j$.
If $u_j=x_{i_{j,1}}\cdots x_{i_{j,t}}$ and $v_j=x_{i'_{j,1}}\cdots x_{i'_{j,t}}$ for any $1\leq j\leq N$, then by \cite[Relation (6.3)]{EH},
\[
i_{1,1}\leq i_{2,1}\leq \cdots \leq\ i_{N,1}\leq i_{1,2}\leq i_{2,2}\leq \cdots\leq i_{N,2} \leq \cdots\leq i_{1,t} \leq i_{2,t} \leq \cdots \leq i_{N,t}
\]
and
\[
i'_{1,1}\leq i'_{2,1}\leq \cdots \leq\ i'_{N,1}\leq i'_{1,2}\leq i'_{2,2}\leq \cdots\leq i'_{N,2} \leq \cdots\leq i'_{1,t} \leq i'_{2,t} \leq \cdots \leq i'_{N,t}.
\]
Since $\deg_{x_r} (u_1\cdots u_N)=\deg_{x_r} (v_1\cdots v_N)$ for $r=1,\ldots,q-1$ with $q\leq n-1$, it follows from the above sequences of inequalities that for any $i_{j,k}\leq q-1$, $i_{j,k}=i'_{j,k}$. Hence
$\deg_{x_r}(u_j)=\deg_{x_r}(v_j)$ for  all $j$ and $1\leq r\leq q-1$. Condition (ii) of Definition \ref{l} implies that there exists $m$ such that $\deg_{x_q}(u_m)<\deg_{x_q}(v_m)$.

Let $u_m=x_{k_1}x_{k_2}\ldots x_{k_t}$, $v_m=x_{l_1}x_{l_2}\ldots x_{l_t}$ such that $k_1<\cdots< k_t$ and $l_1<\cdots< l_t$ and $q=l_i$ for some $1\leq i< t$. Then  $k_1=l_1, \ldots, k_{i-1}=l_{i-1}$ and $k_i>l_i=q$. Set $j=k_i$. We show that $x_qu_m/x_j\in I_t(G)$. By contradiction suppose that $(\supp(u_m)\setminus \{x_j\})\cup\{x_q\}$ is not an independent set of $G$. Then $\{q,k_r\}\in E(G)$ for some $1\leq r\leq t$, $r\neq i$. Since $k_1=l_1, \ldots, k_{i-1}=l_{i-1}$ and $\{x_q,x_{l_1},\ldots,x_{l_{i-1}}\}\subseteq \supp(v_m)$, we have $r>i$. Hence $k_r\in N_{G^q}[q]$, where $G^q=G[q,q+1,\ldots,n]$. Observe that $q<k_i<k_r$ and $N_{G^q}[q]$ is an interval and a clique of $G$. This implies that $k_i\in N_{G^q}[q]$ and $\{k_i,k_r\}\in E(G)$. Since $\{x_{k_i},x_{k_r}\}\subseteq \supp(u_m)$, it follows that $\supp(u_m)$
does not correspond to an independent set of $G$, a contradiction.
 \end{proof}

According to \cite[Theorem 5.1]{HHV}  (see also \cite[Theorem 6.24]{EH}), the Rees ring of a monomial ideal satisfying the $\ell$-exchange property has a particularly nice presentation. To describe this result,  let
 $<_{\lex}$ be the lexicographic order on $S$ with respect to $x_1>\cdots>x_n$. A new monomial order $<_{\lex}'$ on $R$ is defined as follows:
for two monomials  $u_1, u_2\in S$ and  two monomials $v_1,v_2\in T$, we set
$u_1v_1<_{\lex}'u_2v_2$ if and only if
(i) $u_1<_{\lex} u_2$ or
(ii)  $u_1=u_2$ and $v_1<'v_2$.

\begin{Theorem}
\label{generalgb}
Let $I$ be a monomial ideal generated in one degree, satisfying the $\ell$-exchange property. Then the reduced Gr\"obner basis of the toric ideal $J$ with respect to
$<_{\lex}'$ consists of all binomials belonging to the reduced Gr\"obner basis of
$L$  with respect to $<'$ together with  the binomials
\[
x_{i}y_{u}-x_{j}y_{v},
\]
where $x_{i}>x_{j}$  with $x_{i}u=x_{j}v$ and $j$ is the smallest integer   for which $x_{i}u/x_{j}$ belongs to $I$.
\end{Theorem}

Let $I\subset S=K[x_1,\ldots,x_n]$ a graded ideal and $P$ be a prime ideal with $I\subseteq P$.  Recall that   $I$ satisfies the {\em strong persistence property with respect to $P$} if for all $k$ and all  $f\in ((IS_P) ^k:PS_P )\setminus (IS_P)^k$  there exists $g\in IS_P$  such that  $fg\not \in (IS_P)^{k+1}$. The ideal $I$ is said to satisfy the {\em strong persistence property} if it satisfies the strong persistence property with respect to $P$ for any prime ideal $P$ containing $I$.   Note that strong persistence implies persistence, which means that $\Ass(I^k)\subseteq \Ass(I^{k+1})$ for all $k$.

It is shown in \cite[Theorem 1.3]{HQ} that $I$ satisfies strong persistence if and only  if $I^{k+1}:I =I^k$ for all $k$. Under the assumption that $K$ is infinite $I$ satisfies strong persistence if  $\mathcal{R}(I)$ is normal or Cohen-Macaulay, see  \cite[Corollary  1.6]{HQ}.

\medskip
As  a result of  Proposition~\ref{lcondition} and Theorem~\ref{generalgb} we obtain

\begin{Corollary}
\label{finally}
Let $G$ be a proper interval graph. Then for  all $t\geq 2$,  the independence ideal $I_t(G)$ satisfies the strong persistence property  and all of its  powers have linear resolutions.
\end{Corollary}

\begin{proof}
It follows from Proposition~\ref{lcondition} and Theorem~\ref{generalgb} that
all powers of $I_t(G)$ have linear resolution,
see \cite[Corollary 10.1.7]{HH}. By the results of Sturmfels  and Hochster, mentioned already in the proof of Corollary~\ref{sortable*}, we see that $\mathcal{R}(I_t(G))$  is a normal Cohen-Macaulay ring. By \cite[Corollary  1.6]{HQ}, this implies strong persistence.
\end{proof}

The following theorem which is a generalization of Theorem 2.4 in \cite{KM}, proves conjecture 2.3 in \cite{KM} for proper interval graphs.
\begin{Theorem}
Let $G$ be a proper interval graph on the vertex set $[n]$ and $I=I_t(G)$ for some $t\geq 2$. Then for any positive integer $m$, $I^m$ has linear quotients.
\end{Theorem}
\begin{proof}
Note that
$$I=\langle x_{i_1}\cdots x_{i_t}:\ \{i_1,\ldots, i_t\}\in S_t(G)\rangle,$$
where $S_t(G)$ is the set of all $t$-independent subsets of $G$. Firstly we establish

$$\mathcal{A} = \{x_{i_{1,1}}x_{i_{1,2}}\cdots x_{i_{1,m}}
x_{i_{2,1}} x_{i_{2,2}}\cdots x_{i_{2,m}}\cdots x_{i_{t,1}}x_{i_{t,2}}\cdots x_{i_{t,m}}:$$
$$i_{1,1}\leq \cdots \leq i_{1,m} \leq i_{2,1}\leq \cdots\leq i_{2,m}\leq\cdots\leq i_{t,1}\leq\cdots\leq i_{t,m},\ \textrm{and}\ $$ $$\{i_{1,\ell},i_{2,\ell},\ldots,i_{t,\ell}\}\in S_t(G) \ \textrm{for all} \ 1\leq \ell\leq m\ \},$$
is a minimal set of monomial generators for $I^m$.
To this aim, note that any minimal monomial generator $u$ of $I^m$ is the product of $m$ monomials corresponding to some members of $S_t(G)$ and so is of degree $tm$. Thus it can be written as
$u=x_{i_{1,1}}x_{i_{1,2}}\cdots x_{i_{1,m}}x_{i_{2,1}}x_{i_{2,2}}\cdots x_{i_{2,m}}\cdots x_{i_{t,1}}x_{i_{t,2}}\cdots x_{i_{t,m}}$, where $i_{1,1}\leq \cdots\leq i_{1,m} \leq i_{2,1}\leq \cdots\leq i_{2,m}\leq\cdots\leq i_{t,1}\leq\cdots\leq  i_{t,m}$. Assume, in contrary, that there exists an index $1\leq \ell\leq m$ such that $\{i_{1,\ell},i_{2,\ell},\ldots,i_{t,\ell}\}\notin S_t(G)$. Then there exist $1\leq j<j' \leq t$ such that either $i_{j,\ell}=i_{j',\ell}$ or $\{i_{j,\ell}, i_{j',\ell}\}\in E(G)$. Consider the multiset 
$C=\{i_{j,\ell},i_{j,\ell+1},\ldots, i_{j,m},\ldots, i_{j',1}, i_{j',2},\ldots, i_{j',\ell}\}$. Since $G$ is proper interval, any two vertices in $C$ are either equal or adjacent in $G$.
Let $u_1,\ldots,u_m\in \mathcal{G}(I)$ such that $u=\prod_{i=1}^m u_i$. Since $\xb^C|\prod_{i=1}^m u_i$ and $\deg(\xb^C)\geq m+1$, there exists some $1\leq i\leq m$ and $p,q\in C$ such that $x_px_q|u_i$. Note that $u_i$ is squarefree and hence $p$ and $q$ are distinct. Since $\supp(u_i)$ corresponds to an independent set of $G$, we should have $\{p,q\}\notin E(G)$. This is a contradiction.

Conversely, it is obvious that any member of $\mathcal{A}$ is the product of $m$ monomials corresponding to members of $S_t(G)$ and so belongs to $I^m$.

Now consider the lex order induced by $x_1>x_2>\cdots>x_n$ on the minimal monomial generators of $I^m$.
Let $u,u'\in \mathcal{A}$ with $u'>_{lex}u$. Let
$$u=x_{i_{1,1}}x_{i_{1,2}}\cdots x_{i_{1,m}}
x_{i_{2,1}}x_{i_{2,2}}\cdots x_{i_{2,m}}\cdots x_{i_{t,1}}x_{i_{t,2}}\cdots x_{i_{t,m}};$$ and
$$u'=x_{i'_{1,1}}x_{i'_{1,2}}\cdots x_{i'_{1,m}}
x_{i'_{2,1}}x_{i'_{2,2}}\cdots x_{i'_{2,m}}\cdots x_{i'_{t,1}}x_{i'_{t,2}}\cdots x_{i'_{t,m}};$$
such that $i_{1,1}\leq i_{1,2}\leq\cdots \leq i_{1,m} \leq i_{2,1}\leq \cdots\leq i_{2,m}\leq\cdots\leq i_{t,1}\leq\cdots\leq  i_{t,m}$ and $i'_{1,1}\leq i'_{1,2}\leq\cdots \leq i'_{1,m} \leq i'_{2,1}\leq \cdots\leq i'_{2,m}\leq\cdots\leq i'_{t,1}\leq\cdots\leq  i'_{t,m}$  and $\{i_{1,\ell},i_{2,\ell},\ldots,i_{t,\ell}\}\in S_t(G)$ and $\{i'_{1,\ell},i'_{2,\ell},\ldots,i'_{t,\ell}\}\in S_t(G)$ for all $1\leq \ell\leq m$. Let $i_{s,k}$ be the smallest  index such that $i_{s,k}\neq i'_{s,k}$ for some $1\leq s\leq t$ and $1\leq k\leq m$. Then for any $1\leq s'<s$ and any $1\leq k'\leq m$, one has $i_{s',k'}=i'_{s',k'}$ and for any $1\leq k'<k$, $i_{s,k'}=i'_{s,k'}$ and $i'_{s,k}<i_{s,k}$. So, if we set $u''=\frac{x_{i'_{s,k}}}{x_{i_{s,k}}}u$, we have
$u''>_{lex} u$, $u'':u=x_{i'_{s,k}}$ and $x_{i'_{s,k}}|u':u$. Hence it remains to prove that $u''\in \mathcal{A}$. Since $u\in \mathcal{A}$, it is sufficient to show that
$S=\{i_{1,k},i_{2,k},\ldots,i_{s-1,k},i'_{s,k},i_{s+1,k},\ldots,i_{t,k}\}\in S_t(G)$. Suppose in contrary that $S$ is not an independent set of vertices of $G$. Since $u,u'\in \mathcal{A}$, there is an integer $s+1 \leq j \leq t$ such that $\{i'_{s,k}, i_{j,k}\}\in E(G)$. Since $i'_{s,k}< i_{s,k}<i_{j,k}$ and $G$ is proper interval, $\{i_{s,k},i_{j,k}\}\in E(G)$ which contradicts to $u\in \mathcal{A}$. So, we have $S\in S_t(G)$ which implies that $u''\in \mathcal{A}$ as desired.
\end{proof}


\begin{thebibliography}{99}

\bibitem{BW}  K.P. Bogart and D. West, {\it A short proof that ‘proper = unit’}, Discrete Mathematics,  \textbf{201}, no. 1, (1999),  21-23.

\bibitem{BH} W. Bruns and J. Herzog,  Cohen-Macaulay Rings, Cambridge Studies in advanced mathematics \textbf{39}, Cambridge
University Press, Cambridge, UK, 1998.



\bibitem{EH}  V. Ene and J. Herzog,   Gr\"obner Bases in Commutative Algebra, American Mathematical Soc., (2011).

\bibitem{G} F. Gardi, {\it The Roberts characterization of proper and unit interval graphs}, Discrete Mathematics \textbf{307}, no. 22, (2007), 2906-2908.




\bibitem{HH}  J. Herzog and T. Hibi,  Monomial ideals, Graduate Texts in Mathematics. Springer, New York, 2010.


\bibitem{HHO} J. Herzog, T. Hibi and H. Ohsugi,  Binomial ideals,  Graduate Texts in Mathematics. Springer, New York, 2018.

\bibitem{HHV} J. Herzog, T. Hibi and M. Vladoiu, {\it Ideals of fiber type and polymatroids}, Osaka J. Math. \textbf{42},  (2005), 807--829.


\bibitem{HQ}  J. Herzog and A. Qureshi, {\it Persistence and stability properties of powers of ideals}, J. Pure and Appl. Alg.  \textbf{219},  (2015), 530--542.


\bibitem{HRV} J. Herzog, A. Rauf and M. Vladoiu, {\it The stable set of associated prime ideals of a polymatroidal
ideal}, J. Algebraic Combinatorics,  \textbf{37}, no. 2, (2013), 289-312.




\bibitem{KM} F. Khosh-Ahang and  S. Moradi, {\it Some algebraic properties of t-clique ideals}, Communications in Algebra \textbf{47},  (2019), 2870–2882.


\bibitem{LO} P.J. Looges and S. Olariu, {\it Optimal greedy algorithms for indifference graphs}, Comput. Math. Appl. \textbf{25}, (1993), 15–25 .

\bibitem{MMV} J. Martinez-Bernal, S. Morey and R.H. Villarreal, {\it Associated primes of powers of edge ideals}, Collectanea Mathematica \textbf{63}, no. 3, (2012), 361-374.

\bibitem{M}  S. Moradi, {\it $t$-clique ideal and $t$-independence ideal of a graph}, Communications in Algebra \textbf{46},  (2018), 3377--3387.

\bibitem{MRFS} S. Moradi, M. Rahimbeigi, F. Khosh-Ahang and A. Soleyman Jahan, {\it A family of monomial ideals with the persistence property}, Journal of Algebra and its Applications,  \textbf{18}, No. 5, (2019) 1950093.


\bibitem{R1} F.S. Roberts,  Representations of indifference relations, Ph.D. Thesis, Stanford University, Stanford, CA, 1968.


\bibitem {St} B. Sturmfels, Gr\"obner bases and convex polytopes, American Mathematical Society, 1996.
\end{thebibliography}
\end{document}